\documentclass[twoside,12pt,reqno]{article}
\usepackage{amsmath,amsxtra,latexsym,amsthm,amssymb,amscd,amsfonts}
\usepackage[portrait, top=3.5cm, bottom=3cm, left=3.5cm, right=2cm]{geometry}
\usepackage{graphicx}
\usepackage{color}
\usepackage{enumerate}
\usepackage{subcaption}

\theoremstyle{plain}
\newtheorem{lemma}{\bf Lemma}[section]

\newtheorem{remark}{\bf Remark}[section]

\newtheorem{theorem}{\bf Theorem}[section]

\def\tto{\;{\lower 1pt \hbox{$\rightarrow$}}\kern -10pt
\hbox{\raise 2pt \hbox{$\rightarrow$}}\;}

\def\Bar{\overline}

\def\B{I\!\!B}

\def\R{{\rm I\!R}}

\def\gph{\mbox{\rm gph}\,}

\def\ker{\mbox{\rm ker}\,}

\newtheorem{Theorem}{Theorem}[section]

\newtheorem{Example}[Theorem]{Example}

\normalsize

\begin{document}

\date{}
\title{\bf Stability analysis of split equality and split feasibility problems}

\author{Vu Thi Huong\thanks{Digital Data and Information for Society, Science, and Culture,
		Zuse Institute Berlin, 14195 Berlin, Germany; and
		Institute of Mathematics, Vietnam Academy of Science and Technology, 10072 Hanoi, Vietnam. Emails: huong.vu@zib.de; vthuong@math.ac.vn},\quad Hong-Kun Xu\thanks{School of  Science, Hangzhou Dianzi University, 310018 Hangzhou, China; and College of Mathematics and Information Science, Henan Normal University, 453007 Xinxiang,  China. Email: xuhk@hdu.edu.cn. Corresponding author},  \quad  Nguyen Dong Yen\thanks{Institute of Mathematics, Vietnam Academy of Science and Technology, 10072 Hanoi,  Vietnam. Email: ndyen@math.ac.vn}}

\maketitle

\begin{quote}
\noindent {\bf Abstract.}
In this paper, for the first time in the literature, we study the stability of solutions of two classes of feasibility
(i.e., split equality and split feasibility) problems by set-valued and variational analysis techniques. Our idea is to equivalently
 reformulate the feasibility problems as parametric generalized equations to which set-valued and variational analysis techniques apply. Sufficient conditions, as well as necessary conditions, for the Lipschitz-likeness of the involved solution maps are proved by exploiting special structures of the problems and by using an advanced result of B.S. Mordukhovich [J. Global Optim. 28, 347--362 (2004)]. These conditions stand on a solid interaction among all the input data by means of their dual counterparts, which are transposes of matrices and regular/limiting normal cones to sets. Several examples are presented to illustrate how the obtained results work in practice and also show that the existence of nonzero solution assumption made in the necessity conditions cannot be lifted.



\end{quote}

\section{Introduction}\label{Sect_Introduction} \setcounter{equation}{0}
 Let two nonempty closed convex sets $C \subset \mathbb R^n$ and $Q \subset \mathbb R^ m$ and a matrix $A \in \mathbb R^{m\times n}$ be given.
As a feasibility problem, the \textit{split feasibility problem} (SFP) is to find a point $x$ with the property
$$x \in C \quad  \mbox{and}  \quad Ax\in Q.$$
This problem was introduced by Censor and Elfving~\cite[Section~6]{Censor_Elfving_94} in 1994 to model phase retrieval and other image restoration problems in signal processing. The model has been received a lot of attention, due to its wide range of applications in diverse applied areas such as  signal processing/imaging reconstruction~\cite{Byrne02,Byrne04,Qu_Liu_18,Shehu_etal_19}, medical treatment of intensity-modulated radiation therapy~\cite{Censor_etal_05,Censor_etal_06,Censor_etal_08,Censor_etal_10}, and gene regulatory network inference~\cite{Wang_etal17}.

Another class of feasibility problems, referred to as the
\textit{split equality problem} (SEP), is to find a pair of points $x, y$ with the property
$$x \in C,\  \,  y \in Q \quad  \mbox{and}  \quad Ax = By,$$
where again $C\subset \mathbb{R}^n$ and $Q \subset \mathbb{R}^m$ are given nonempty closed convex sets, and $A\in \mathbb{R}^{l\times n}$ and $B\in\mathbb{R}^{l \times m}$ are given matrices. This extension of (SEP)  was proposed by Moudafi and his coauthors~\cite{Moudafi_Shemas13,Moudafi13,Moudafi14,Byrne_Moudafi_17}. The study of (SEP) was motivated by the applications in decomposition methods for PDEs, game theory, and intensity-modulated radiation therapy; see Moudafi~\cite{Moudafi13}. Clearly, when $l = m$ and $B$ is the identity matrix in $\mathbb{R}^{m \times m}$, (SEP) recovers (SFP). Conversely, (SEP) can be converted to a special case of (SFP) via a product space approach (see, e.g.,~\cite[Section~3]{Xu_Celgieski21}).

During the last nearly three decades, many efforts have been made to design solution algorithms for (SFP) (see~\cite{Censor_Elfving_94,Byrne02,Byrne04,Qu_Liu_18,Shehu_etal_19,Wang_etal17,Xu06,Censor_etal_07,Xu10,Lopez_etal_12,Dong_et_al18,Chen_et_al20,Dong_et_al21,Che_et_al22} and the references therein), (SEP) (see, for instance,~\cite{Moudafi_Shemas13,Moudafi13,Moudafi14,Byrne_Moudafi_17,Vuong_etal_15,Dong_etal_15,Reich_Tuyen_21,Xu_Celgieski21}), and related problems (see, e.g.,~\cite{He_etal_16,HeXu17,Long_etal_19,YenLH_etal_19}). In many papers, (SFP) and (SEP) are reformulated as constrained optimization problems with smooth objectives. This technique allows one to effectively exploit advanced tools from optimization theory and operator machinery.

In practical applications, data of a mathematical model in question might be given approximately or undergo perturbations. Thus, studying the \textit{stability of solutions} with respect to small changes of data is an important topic. As an example, let $\bar x$ be a solution of the linear equation $Ax = b$, which is a special simple form of (SFP) with $C= \R^{n}$ and $Q=\{b\}$. When $(A, b)$ is slightly perturbed to $(A', b')$, the following questions arise in a natural way (see, e.g.,~\cite[Appendix A]{Rusz06} and \cite{Robinson75}):
\begin{itemize}
\item[{\rm (i)}] \textit{Does the perturbed equation $A'x = b'$ still possess a solution}?
\item[{\rm (ii)}] \textit{If such solutions exist, whether there is a solution $x'$ ``near" $\bar x$}?
\end{itemize}
  These kinds of concerns lead to the concepts of ``\textit{stable}" and  ``\textit{well-posed}" problems, in contrast with ``\textit{ill-possed}" problems; see Tikhonov and Arsenin~\cite[Chapter~1]{Tikhonov_Arsenin_77}. For ill-posed problems, regularization methods are applied to produce approximate solutions~\cite[Chapter~2]{Tikhonov_Arsenin_77}. Meanwhile, if a problem is enough stable, one can design effective solution algorithms with fast convergence rates. This is the case of the Newton method applied to generalized equations~\cite{Dontchev96}, where the related set-valued map is Lipschitz-like at the reference point in its graph. More details about interrelations between stability properties and solution methods for generalized equations, optimization problems, linear programs can be found in the recent papers~\cite{Klatte_Kummer_09,Canovas_etal_16,Wang_etal_22}.

However, to the best of our knowledge, there has been no study on stability analysis on solutions of (SFP) and (SEP) so far,
 and it is the purpose of this paper to study this gap. More precisely,
 we shall study the stability of solutions of (SEP) and (SFP) when the matrices $A$ and $B$ are perturbed. Define the solution map of (SEP) as the set-valued map $S: \mathbb{R}^{l\times n} \times \mathbb{R}^{l\times m} \rightrightarrows \mathbb{R}^n \times \mathbb{R}^m$ by
  \begin{equation*}
  	S (A, B):= \{(x, y)  \in C \times Q \, : \, Ax=By\}, \quad (A, B) \in \mathbb{R}^{l\times n} \times \mathbb{R}^{l\times m}
  \end{equation*}
and the solution map of (SFP) as the map $S_1(\cdot): \mathbb R^{m \times n} \rightrightarrows \mathbb R^n$ by
\begin{equation*}
	S_1(A):=\{x \in C \, : \, Ax \in Q\}, \quad A \in \mathbb R^{m\times n}.
\end{equation*}
Several basic properties of the set-valued maps $S(\cdot)$ and $S_1(\cdot)$ are deserved considerations; these include upper semicontinuity, lower semicontinuity, inner semicontinuity, continuity, calmness, Lipschitz-likeness, Lipschitz continuity (see, e.g,~\cite{Aubin_Frank_90,B-M06} and \cite[Section~3.1]{Mordukhovich_2018} for the roles of these concepts in set-valued and variational analysis). \textit{We will focus on the Lipschitz-likeness property of $S(\cdot)$ and $S_1(\cdot)$, a robust and maybe the most significant property among the just mentioned ones.}

Let $\Phi: X \rightrightarrows  Y$ be a set-valued map between normed spaces $X, Y$ and let $(\bar x, \bar y) \in X \times Y$ be such that $\bar y \in \Phi (\bar x)$. One says that $\Phi$  is  \textit{Lipschitz-like} (\textit{pseudo-Lipschitz}, or has the \textit{Aubin property}) at $(\bar x, \bar y)$ if there exist neighborhoods $U$ of $\bar x$, $V$ of $\bar y$, and constant $\ell > 0$ such that
\begin{align*}
	\Phi(x') \cap V \subset \Phi (x) + \ell \| x' - x\| \Bar {\B}_{\mathbb{R}^m}, \quad \forall x',\, x \in U,
\end{align*}
where $\Bar {\B}_{\mathbb{R}^m}$ signifies the closed unit ball in $\mathbb{R}^m$. The concept was originally introduced by Aubin~\cite{Aubin84} in 1984 to study the stability of the solution set of a convex optimization problem. Since then, it has played a very important role in optimization theory and variational analysis; see, for instance, the comments in the recent papers~\cite{Huyen_Yen16,Li_Ng_18}. Clearly, the Lipschitz-likeness of $\Phi$ at $(\bar x, \bar y)$ implies the following:
\begin{itemize}
\item[{\rm (a)}] For every $x$ ``\textit{near}'' $\bar x$ (to be made more precise later, $x \in U$), the set $\Phi (x)$ is nonempty;
\item[{\rm (b)}] For every $x$ ``\textit{near}'' $\bar x$, there exists $y \in \Phi(x)$ such that the distance between $y$ and $\bar y$ is not greater than $\ell$ times the distance between $x$ and $\bar x$. This means that $y$ is ``\textit{near}" $\bar y$ and the distance is quantified by the bound $\ell \|x - \bar x\|$;
\item[{\rm (c)}] $\Phi$ is inner semicontinuous at $(\bar x,\bar y)$ (see~\cite[p.~71]{B-M06} and~\cite[Proposition~3.1]{Yen87}).
\end{itemize}
Thus, among other things, the Lipschitz-likeness property of $S(\cdot)$ and $S_1(\cdot)$ provides us with affirmative answers to the fundamental questions of the above-mentioned types (i) and (ii) on the solution stability of (SEP) and (SFP).

The Lipschitz-likeness of a set-valued map can be studied by various tools from variational analysis such as metric regularity, generalized derivatives and coderivatives. We choose the approach involving the limiting coderivative in Mordukhovich's theory to investigate the Lipschitz-likeness of $S(\cdot)$ and $S_1(\cdot)$. Namely, we represent the solution map of (SEP) as the one of a parametric generalized equation. Then, thanks to the special structure of the problem and the fundamental result of Mordukhovich~\cite[Theorem 4.2(ii)]{Mor04}, we are able to establish a sufficient condition and a necessary condition for the Lipschitz-likeness of~$S(\cdot)$. The characterization (see condition~\eqref{Lipchitz-like_S_con} in the next section) expresses a solid interaction among all the input data, by means of their \textit{dual} counterparts: transposes of matrices and normal cones to sets.  For the map $S_1(\cdot)$, a bit different approach is needed. The sufficiency for the Lipschitz-likeness of  $S_1(\cdot)$ is established from the one of $S(\cdot)$ by making use of the close relationship between (SFP) and (SEP). This technique, however, does not work for the necessity assertion. The reason lies in the incompatibility between the domains of the maps $S(\cdot)$ and $S_1(\cdot)$. We, therefore, again lean on \cite[Theorem 4.2(ii)]{Mor04} to establish the necessity (condition~\eqref{Lipchitz-like_S1_con} in Section~\ref{SFP}) for the Lipschitz-likeness of  $S_1(\cdot)$. Several examples have been designed to illustrate how the obtained results (Theorems~\ref{Lipchitz-like_S_thm} and \ref{Lipchitz-like_S1_thm}) work in practice and
also show that the existence of nonzero solution assumption made in the necessity conditions cannot be lifted.

The rest of the paper is structured as follows. In Section~\ref{Sect_Preliminaries}, several concepts and tools from set-valued and variational analysis are recalled. Our main results on the Lipschitz-likeness of the solution maps $S(\cdot)$ and $S_1(\cdot)$ are presented in Sections~\ref{SEP} and~\ref{SFP}, respectively. Finally, a summary is included in Section \ref{Sect_Conclusions}.

Throughout the paper, if $M\in\mathbb R^{q\times p}$ is a matrix, then $M^{\rm T}$ stands for the transpose of $M$. The kernel is defined by ${\rm ker}\, M=\big\{ x\in\mathbb R^p\,:\, Mx=0\big\}$. The inverse of a set $\Omega\subset\mathbb R^q$ via the operator $M:\mathbb R^p\to\mathbb R^q$  is defined by $M^{-1}(\Omega)=\big\{ x\in\mathbb R^p\,:\, Mx\in\Omega\big\}$. The space of the linear operators from $\mathbb R^p$ to $\mathbb R^q$ is denoted by $L\big(\mathbb R^p,\mathbb R^q\big)$. By ${\rm int}\,\Omega$ we denote the topological interior of $\Omega$. Let  ${\B}(\bar y,\rho)$ (resp., $\Bar {\B}(\bar y,\rho)$) stand for the open ball (resp., the closed ball) with center $\bar y\in\mathbb R^q$ and radius $\rho>0$. The closed unit ball in $\mathbb{R}^q$ is abbreviated as $\Bar{\B}_{\mathbb{R}^q}$.

\section{Preliminaries}\label{Sect_Preliminaries} \setcounter{equation}{0}
Some basic concepts and tools from set-valued and variational analysis~\cite{B-M06,Mordukhovich_2018}, which will be needed in the sequel, are recalled in this section.

Let  $\Phi:\mathbb R^n\rightrightarrows  \mathbb R^m$ be a set-valued map. The \textit{graph} of $\Phi$ is the set $$\gph \Phi:=\big\{(x,y)\in\mathbb R^n\times\mathbb R^m\,:\, y\in\Phi(x)\big\}.$$
We say that $\Phi$ has \textit{closed graph} if $\gph \Phi$ is closed in the product space $\mathbb R^n\times\mathbb R^m$, which is endowed with the norm $\|(x,y)\|:=\|x\|+\|y\|$ for all $(x,y)\in\mathbb R^n\times\mathbb R^m$. For any $(\bar x,\bar y)\in\gph\Phi$, one says that $\Phi$ is {\it locally closed} around $(\bar x,\bar y)$ if there exists $\rho>0$ such that $\big(\gph\Phi\big)\cap \Bar {\B}((\bar x, \bar y),\rho)$ is closed in $\mathbb R^n\times\mathbb R^m$.  If $\Phi$ has closed graph, then it is locally closed around any point in its graph.

\medskip
Let $\Omega$ be a nonempty subset of $\mathbb{R}^n$ and $\bar x \in \Omega$. The {\it regular normal cone} (or \textit{Fr\'echet normal cone}) to $\Omega$ at $\bar x$ is defined by
\begin{align*}
	\widehat N(\bar x; \Omega)=\Big\{ x'\in \mathbb{R}^n \, :\, \limsup\limits_{x \xrightarrow{\Omega}\bar x} \dfrac{\langle x', x-\bar x \rangle}{\|x-\bar x\|} \leq 0 \Big\},
\end{align*}
where $x \xrightarrow{\Omega} \bar x$ means that $x \rightarrow \bar x$ and $ x\in \Omega$. The {\it limiting normal cone} (or {\it Mordukhovich normal cone}) to $\Omega$ at $\bar x$ is given by
\begin{eqnarray*}\begin{array}{rl}
N(\bar x;\Omega)=\big\{x'\in \mathbb{R}^n\, :\,  \exists \mbox{ sequences } x_k\to \bar x,\ x_k'\rightarrow x'
	\quad \mbox {with } x_k'\in \widehat
		N(x_k;\Omega)\, \mbox{ for all }\, k=1,2,\dots\big\}.\end{array}\end{eqnarray*}
We put $\widehat N(\bar x; \Omega)=N(\bar x; \Omega) =\emptyset$ if $\bar x \not\in \Omega$. One has $\widehat N(\bar x; \Omega) \subset N(\bar x; \Omega)$ for all $\Omega \subset \mathbb{R}^n$ and $\bar x \in \Omega$.  By~\cite[Prop. 1.7]{Mordukhovich_2018}), if  $\Omega$ is convex, then the regular normal cone and the limiting normal cone to $\Omega$ at $\bar x\in\Omega$ coincide, and both cones reduce to the \textit{normal cone in the sense of convex analysis}, i.e.,
\begin{align*}
	\widehat{N}(\bar x; \Omega)= {N}(\bar x; \Omega)=\big\{x'\in \mathbb{R}^n \, :\, \langle x', x-\bar x \rangle \le 0, \ \forall x\in\Omega\big\}.
\end{align*}

Let $(\bar x,\bar y)$ belong to the graph of a set-valued map $\Phi:\mathbb R^n\rightrightarrows  \mathbb R^m$. The set-valued map $D^*\Phi(\bar x,\bar y):\mathbb R^m\rightrightarrows \mathbb R^n$ defined by
\begin{eqnarray*}
		D^*\Phi(\bar x,\bar y)(y'):=\big\{x'\in \mathbb R^n\, :\, (x',-y')\in N((\bar x,\bar y);\mbox{gph}\,\Phi)\big\},\quad y'\in\mathbb{R}^m
\end{eqnarray*}
 is called the {\it limiting coderivative} (or {\it Mordukhovich coderivative}) of the set-valued map $\Phi$ at $(\bar x,\bar y)$.  When $\Phi$ is single-valued and $\bar y=\Phi(\bar x)$, we write $D^*\Phi(\bar x)$ for $D^*\Phi(\bar x,\bar y)$. If $\Phi:\mathbb R^n\to \mathbb R^m$ is \textit{strictly differentiable} \cite[p.~19]{B-M06} at $\bar x$ (in particular, if $\Phi$ is continuously Fr\'echet differentiable in a neighborhood of $\bar x$) with derivative $\nabla\Phi(\bar x)$, then
\begin{eqnarray*}
	 D^*\Phi(\bar x)(y')=\big\{\nabla\Phi(\bar x)^*y'\big\}, \quad \forall	y'\in\mathbb R^m.
 \end{eqnarray*}
Here, the {\it adjoint operator} $\nabla\Phi(\bar x)^*$ of $\nabla\Phi(\bar x)$ is defined by setting
$$\langle\nabla\Phi(\bar x)^*y',x\rangle=\langle y',\nabla\Phi(\bar x)x\rangle$$
for every $x\in\mathbb R^n$ (see \cite[Theorem 1.38]{B-M06} for more details). In accordance with the definition given at~\cite[p.~351]{Mor04}, one says that $\Phi$ is \textit{graphically regular} at $(\bar x,\bar y)$ if $$N((\bar x,\bar y);\mbox{gph}\,\Phi)=\widehat{N}((\bar x,\bar y);\mbox{gph}\,\Phi).$$

 Recall that a set-valued map $\Phi:\mathbb R^n\rightrightarrows  \mathbb R^m$ is \textit{Lipschitz-like} (\textit{pseudo-Lipschitz}, or has the \textit{Aubin property}; see \cite{Aubin84}) at $(\bar x, \bar y) \in \gph\Phi$ if there exist neighborhoods $U$ of $\bar x$, $V$ of $\bar y$, and constant $\ell > 0$ such that
\begin{align*}
	\Phi(x') \cap V \subset \Phi (x) + \ell \| x' - x\| \Bar {\B}_{\mathbb{R}^m}, \quad \forall x',\, x \in U.
\end{align*}

Let $X, Y, Z$ be finite-dimensional spaces.  Consider a \textit{parametric generalized equation}
\begin{equation}\label{para. generalized equa. def.}
	0\in f(x, y)+G(x, y)
\end{equation}
with the decision variable $y$ and the parameter $x$, where $f:X\times Y\rightarrow Z$ is a single-valued map, and $G: X\times Y\rightrightarrows Z$ is a set-valued map. The \textit{solution map} of \eqref{para. generalized equa. def.} is the set-valued map $S: X \rightrightarrows Y$ defined by
\begin{equation}\label{solution map def.}
	S(x):=\left\{y\in Y \;:\; 0\in f(x, y)+G(x, y)\right\}, \quad x\in X.
\end{equation}

The next theorem, which is a special case of~\cite[Theorem 4.2(ii)]{Mor04} states a necessary and sufficient condition for the Lipschitz-likeness property of the solution map $S$ defined in \eqref{solution map def.}.

\begin{theorem}{\rm (See~\cite[Theorem 4.2(ii)]{Mor04})}\label{Thm. 4.2-Mor-2004-JOGO}
	Let $X, Y, Z$ be finite-dimensional spaces and let $(\bar x, \bar y)$ satisfy \eqref{para. generalized equa. def.}. Suppose that $f$ is strictly differentiable at $(\bar x, \bar y)$, $G$ is locally closed around  $(\bar x, \bar y,\bar z)$ with $\bar z:=-f(\bar x, \bar y)$ and, moreover, $G$ is graphically regular at $(\bar x, \bar y,\bar z)$. If
	\begin{equation}\label{Lipschitz-like condition}
		\big [(x', 0)\in \nabla f(\bar x, \bar y)^*(z')+D^*G(\bar x, \bar y, \bar z)(z') \big]\Longrightarrow [x'=0,\, z'=0],
	\end{equation} then $S$ is Lipschitz-like at $(\bar x, \bar y)$. Conversely, if $S$ is Lipschitz-like at $(\bar x, \bar y)$ and if the regularity condition
	\begin{equation}\label{regularity condition}
		\big [(0,0)\in \nabla f(\bar x, \bar y)^*(z')+D^*G(\bar x, \bar y, \bar z)(z') \big]\Longrightarrow [z'=0]
\end{equation}
	is satisfied, then one must have~\eqref{Lipschitz-like condition}.	
\end{theorem}

\noindent{\em Remark.}
Observe that the characterization~\eqref{Lipschitz-like condition} and the regularity condition~\eqref{regularity condition} are described in terms of the limiting coderivative of the input data $f$ and $G$. The interested reader is referred to Subsection~4.2~[4B] in the book by Dontchev and Rockafellar~\cite{Dontchev_Rockafellar14} for an alternative approach using ``\textit{graphical derivative}'' criteria. If the set-valued map $G$  is constant (as it will be in the next two sections), one can consult~\cite[Theorem 4F.5]{Dontchev_Rockafellar14} for a simpler characterization of the Lipschitz-likeness of~$S$. However, there is no free lunch. Theorem~ 4F.5~\cite{Dontchev_Rockafellar14} requires the
\textit{ample parameterization condition}, which means that the partial derivative of the function $f$ with respect to the parameter~$x$ at the reference point is \textit{surjective}. Therefore, this theorem cannot be applied to the problems to be considered the next sections; see the functions in formulas~\eqref{f} and~$\eqref{f1}$ in Sections \ref{SEP} and \ref{SFP}, respectively.

\section{Split Equality Problems}\label{SEP}\setcounter{equation}{0}
The \textit{split equality problem} (SEP) given by two nonempty closed convex sets $C\subset \mathbb{R}^n$ and $Q \subset \mathbb{R}^m$, and two matrices $A\in \mathbb{R}^{l\times n}$ and $B \in \mathbb{R}^{l \times m}$ is to find a pair of points $(x, y)$ with the property:
$(x, y)\in C \times Q$ and $Ax = By$. Our purpose of this section is to establish some stability properties of (SEP). More precisely, we will focus on local continuity properties (including the Lipschitz-likeness) of the solution map $S: \mathbb{R}^{l\times n} \times \mathbb{R}^{l\times m} \rightrightarrows \mathbb{R}^n \times \mathbb{R}^m$ defined by
\begin{equation}\label{S}
S (A, B):= \{(x, y)  \in C \times Q \, : \, Ax=By\}, \quad (A, B) \in \mathbb{R}^{l\times n} \times \mathbb{R}^{l\times m}.
\end{equation}
Thus, the matrices $A$ and $B$ are subject to perturbations, while the constraint sets $C$ and $Q$ are fixed.

Our main result in this section, which gives sufficient and necessary conditions for the Lipschitz-likeness property of the solution map
$S$ defined in~\eqref{S} at a given point in its graph, reads as follows.

\begin{theorem}\label{Lipchitz-like_S_thm}
	Let $(\bar A, \bar B) \in \mathbb{R}^{l\times n} \times \mathbb{R}^{l\times m}$ and $(\bar x, \bar y) \in S (\bar A,\bar B)$. Suppose that
	\begin{equation}\label{Lipchitz-like_S_con}
		\big(\bar A^{\rm T}\big)^{-1}\big(-N(\bar x; C)\big) \cap \big(\bar B^{\rm T}\big)^{-1}\big(N(\bar y; Q)\big)  = \{0\}.
	\end{equation}
Then,  the solution map $S(\cdot)$ of {\rm (SEP)} is Lipschitz-like at $\big((\bar A, \bar B),(\bar x, \bar y)\big)$. Conversely, if $S(\cdot)$ is Lipschitz-like at $\big((\bar A, \bar B),(\bar x, \bar y)\big)$ and $(\bar x, \bar y) \neq (0, 0)$, then~\eqref{Lipchitz-like_S_con} holds.
\end{theorem}

\begin{remark}{\rm Clearly, if either $(\bar A^{\rm T})^{-1}(-N(\bar x; C)) =\{0\}$ or $(\bar B^{\rm T})^{-1}(N(\bar y; Q))=\{0\}$, then condition~\eqref{Lipchitz-like_S_con} is satisfied. So, if either $\bar x \in \mbox{int}\, C$ and $\ker \bar A^{\rm T} = \{0\}$ or $\bar y \in \mbox{int}\, Q$ and $\ker \bar B^{\rm T} = \{0\}$, then~\eqref{Lipchitz-like_S_con} holds true.}
\end{remark}

As an illustration for the applicability of the first assertion of Theorem~\ref{Lipchitz-like_S_thm}, let us consider the next example.

\begin{Example}{\rm Choose $l=m=n=1$, $C= [-1, 1]$, and $Q=[0, +\infty)$. Besides, let $\bar A=(\bar\alpha)$ for some $\bar\alpha> 0$ and $\bar B= (1)$. Fix any $\bar x \in [0, 1]$ and let $\bar y= \bar\alpha \bar x$. It is easy to verify that $(\bar x, \bar y) \in S(\bar A, \bar B)$. If $\bar x\in [0, 1)$, then $\bar x \in \mbox{int}\, C$. Since $\ker \bar A^{\rm T} = \{0\}$, this implies that the condition ~\eqref{Lipchitz-like_S_con} is satisfied. Therefore, by Theorem~\ref{Lipchitz-like_S_thm}, $S(\cdot)$ is Lipschitz-like at $\big((\bar A, \bar B), (\bar x, \bar y)\big)$. If $\bar x = 1$, then $\bar y = \bar\alpha\in \mbox{int}\, Q$. As $\ker \bar B^{\rm T} = \{0\}$, the latter implies that~\eqref{Lipchitz-like_S_con} is fulfilled. Hence we get the Lipschitz-likeness of $S(\cdot)$ at $\big((\bar A, \bar B), (\bar x, \bar y)\big)$ by Theorem~\ref{Lipchitz-like_S_thm}.}
\end{Example}

The following example is designed to show how the second assertion of Theorem~\ref{Lipchitz-like_S_thm} can be useful for analyzing  concrete problems.

\begin{Example}{\rm Consider the problem (SEP) with $l=m=n=1$, $C= [-1, 1]$, and $Q=[0, +\infty)$. Let $\bar A=(0)$, $\bar B= (1)$. Take any $\bar x \in [-1, 1]$ and put $\bar y =0$. Then, one has $(\bar x, \bar y) \in S(\bar A, \bar B)$. Note that $\big(\bar A^{\rm T}\big)^{-1}\big(-N(\bar x; C)\big) = \mathbb R$ while $$\big(\bar B^{\rm T}\big)^{-1}\big(N(\bar y; Q)\big) = N(\bar y; Q)  = (-\infty, 0].$$ Thus, $$\big(\bar A^{\rm T}\big)^{-1}\big(-N(\bar x; C)\big) \cap \big(\bar B^{\rm T}\big)^{-1}\big(N(\bar y; Q)\big)  = (-\infty, 0].$$
Consequently, the condition~\eqref{Lipchitz-like_S_con} is invalid. Therefore, by Theorem~\ref{Lipchitz-like_S_thm} we can infer that $S(\cdot)$ is \textit{not} Lipschitz-like at $\big((\bar A, \bar B), (\bar x, \bar y)\big)$, provided that $\bar x \neq 0$.}
This shows that the nonzero solution (i.e., $(\bar x, \bar y) \neq (0, 0)$) assumption in the necessity condition of Theorem \ref{Lipchitz-like_S_thm} can't be lifted.
\end{Example}

 Our main idea of the proof of Theorem~\ref{Lipchitz-like_S_thm} is to transform (SEP) to a special type of generalized equations \eqref{para. generalized equa. def.} to which the fundamental result recalled in Theorem~\ref{Thm. 4.2-Mor-2004-JOGO} applies. We proceed as follows.

Let $W = \mathbb{R}^{l\times n} \times \mathbb{R}^{l\times m}$, $U = \mathbb{R}^n \times \mathbb{R}^m$, and $V= \mathbb{R}^n \times \mathbb{R}^m \times \mathbb{R}^l$. Consider the function $f: W\times U \to V$ given by
\begin{equation}\label{f}
	f(w, u):= (-x, -y, Ax - By), \quad w=(A, B) \in  W,\ u=(x, y) \in U
\end{equation}
and the set-valued map $G: W\times U \rightrightarrows V$ with
\begin{equation}\label{G}
	G(w, u):= C \times Q \times \{0_{\mathbb{R}^l}\}, \quad w=(A, B) \in  W,\  u=(x, y) \in U.
\end{equation}
Then, the solution map $(A,B)\mapsto S(A,B)$ of (SEP) coincides with the solution map of the generalized equation
$0 \in f(w, u) + G(w, u)$, i.e., one can write
\begin{equation}\label{S_solution map}
	S(w) = \left\{u \in U \; : \; 0 \in f(w, u) + G(w, u)\right\}, \quad w \in W.
\end{equation}

We will need two lemmas related to~\eqref{S_solution map}: the first one establishes  some properties of the function $f$ and the second one gives a coderivative formula for the set-valued map $G$.

\begin{lemma}\label{lemma_f}
The function $f: W\times U \to V$ in~\eqref{f} is strictly differentiable at any point $(\bar w, \bar u) \in W \times U$ with $\bar w = (\bar A, \bar B)$ and $\bar u =(\bar x, \bar y)$. The derivative $\nabla f(\bar w, \bar u) : W \times U \to V$  of $f$ at $(\bar w, \bar u)$ is given by
\begin{equation}\label{deri_f}
\nabla f(\bar w, \bar u)(w, u) = (-x, -y, \bar A x - \bar B y + A \bar x - B \bar y),\ \,  w = (A, B) \in W,\ u =(x, y) \in U.
\end{equation}
In addition, if $\bar u \neq (0, 0)$, then  the operator $\nabla f(\bar w, \bar u) : W \times U \to V$ is surjective and thus its adjoint operator $\nabla f(\bar w, \bar u)^* : V \to W \times U$ is injective.
\end{lemma}
\begin{proof}
Fix any $(\bar w, \bar u) \in W \times U$ with $\bar w = (\bar A, \bar B)$, $\bar u =(\bar x, \bar y)$. Consider the continuous linear operator $T(\bar w, \bar u): W \times U \to V$ defined by
\begin{equation*}\label{T}
	T(\bar w, \bar u)(w, u) := (-x, -y, \bar A x - \bar B y + A \bar x - B \bar y),\ \, w = (A, B) \in W,\ u =(x, y) \in U.
\end{equation*}
One has
\begin{align*}\label{Frechet deri. calculate}
L:&=\displaystyle\lim_{(w, u)\rightarrow (\bar w, \bar u)}
        \dfrac{f(w, u)-f(\bar w,\bar u)- T(\bar w,\bar u)((w, u)-(\bar w,\bar u))}{\|(w, u)-(\bar w,\bar u)\|}\\
&=\displaystyle\lim_{(w, u)\rightarrow (\bar w, \bar u)}\Big[
		\dfrac{(-x, -y, Ax - By)- (-\bar x, -\bar y, \bar A \bar x - \bar B \bar y)}{\|(w-\bar w, u-\bar u)\|}\\
& \qquad  -\dfrac{\big( -(x -\bar x), -(y -\bar y), \bar A (x-\bar x) - \bar B(y-\bar y)+ (A -\bar A)\bar x -(B-\bar B)\bar y\big)}
        {\|(w-\bar w, u-\bar u)\|}\Big] \\
&= \displaystyle\lim_{(w,u)\rightarrow (\bar w, \bar u)}\bigg(0,0,\dfrac{(A -\bar A)(x-\bar x)-(B -\bar B)(y-\bar y)}{\|(w-\bar w, u-\bar u)\|} \bigg).
\end{align*}
Since
\begin{align*}
\dfrac{\|(A - \bar A)(x-\bar x) - (B -\bar B)(y-\bar y)\|}{\|(w-\bar w, u-\bar u)\|}
& \leq \dfrac{\|A -\bar A\| \|x-\bar x\|}{\|(w-\bar w, u-\bar u)\|} +\dfrac{\|B -\bar B\| \|y -\bar y\|}{\|(w-\bar w, u-\bar u)\|}\\
&\leq \dfrac{\|A -\bar A\| \|x -\bar x\|}{\|A -\bar A\|} +\dfrac{\|B -\bar B\| \|y -\bar y\|}{\|B -\bar B\|}\\
& = \|x -\bar x\| + \|y -\bar y\|=\|u-\bar{u}\|\to 0
\end{align*}
when $u$ tends to $\bar u$, one can easily infer that $L = 0$. This means that $f$ is Fr\'echet differentiable at $(\bar w, \bar u)$ and the derivative $\nabla f(\bar w, \bar u)$ coincides with $T(\bar w, \bar u)$. As the function $\Phi:W\times U\to L(W\times U,V)$ with $\Phi(\bar w,\bar u):=T(\bar w, \bar u)$ for every $(\bar w, \bar u)\in W\times U$ is continuous on $W \times U$, $f$ is strictly differentiable at $(\bar w, \bar u)$ (see~\cite[p.~19]{B-M06}) and its strict derivative is given by~\eqref{deri_f}.

To prove the last assertion of the lemma, suppose $\bar u \neq (0, 0)$. We will establish the surjectivity of the operator $\nabla f(\bar w, \bar u)$ by proving that the system
\begin{equation}\label{equation_new1}
	\nabla f(\bar w, \bar u)(w, u) = v
\end{equation}
has a solution $(w,u)\in W\times U$ with $w=(A,B)\in W$ and $u=(x,y)\in U$ for each given $v\in V$.
To see this, we write $v=(v^1, v^2, v^3)$.
By~\eqref{deri_f} we get
\begin{align*}
\nabla f(\bar w, \bar u)(w, u) = v & \, \Longleftrightarrow \, (-x, -y, \bar A x - \bar B y + A \bar x - B \bar y) = (v^1, v^2, v^3)\\
& \, \Longleftrightarrow \, \begin{cases}
x = -v^1\\
y = -v^2\\
\bar A x - \bar B y + A \bar x - B \bar y = v^3.
\end{cases}
\end{align*}
Thus, the condition~\eqref{equation_new1} forces $u=(x,y)=(-v^1, -v^2)$ and the components $A, B$ of $w$ to satisfy the relation
\begin{equation}\label{identity}
A \bar x - B \bar y = v^3 + \bar A v^1 - \bar B v^2.
\end{equation} For our convenience, we put $\widetilde v^3=v^3 + \bar A v^1 - \bar B v^2$ and let $\widetilde v^3=(\widetilde v^3_1, \widetilde v^3_2, \dots, \widetilde v^3_l)$. Then, \eqref{identity} can be rewritten equivalently as
\begin{equation}\label{identity_new1}
	A \bar x - B \bar y = \widetilde v^3.
\end{equation}
Since the vector $\bar u=(\bar x,\bar y)$ is nonzero and (SEP) is symmetric w.r.t. to $x$ and $y$, $A$ and $B$, $C$ and $Q$ (see~\eqref{S}), we can assume, with no loss of generality, that $\bar x\neq 0$. Suppose that $\bar x= (\bar x_1, \bar x_2, \dots, \bar x_n)$ and $\bar x_{\tilde j} \neq 0$ for some $\tilde j \in \{1,\dots, n\}$. Choose $B= (0) \in \mathbb{R}^{l\times m}$ and define a matrix $A = (a_{ij})\in \mathbb{R}^{l\times n}$ by setting $a_{i\tilde j} = \dfrac{\widetilde v^3_i}{\bar x_{\tilde j}}$ for $i \in \{1,\dots, l\}$, $a_{ij} = 0$ for $i \in \{1,\dots, l\}$ and $j \in \{1,\dots, n\} \setminus \{\tilde j\}$. Then we have
 \begin{align*}
	A \bar x - B \bar y = A \bar x = \left (\begin{array}{c} a_{11}\bar x_1
		+ a_{12} \bar x_2 + \cdots + a_{1n} \bar x_n \\
		a_{21}\bar x_1 + a_{22} \bar x_2 + \cdots + a_{2n} \bar x_n\\
		\vdots\\
		a_{l1}\bar x_1 + a_{l2} \bar x_2 + \cdots + a_{ln} \bar x_n
	\end{array}\right) = \left(\begin{array}{c}\widetilde v^3_1\\
		\widetilde v^3_2\\
		\vdots\\
		\widetilde v^3_l \end{array}\right).
\end{align*}
This means that~\eqref{identity_new1} is fulfilled with the chosen matrices $A$ and $B$.
We have thus proved the existence of a solution of the system \eqref{equation_new1} and hence
the surjectivity of $\nabla f(\bar w, \bar u)$ follows.
Therefore, by using~\cite[Lemma 1.18]{B-M06} we obtain the injectivity of the adjoint operator $\nabla f(\bar w, \bar u)^*$.
\end{proof}

\begin{lemma}\label{lemma_G}
The set-valued map $G: W\times U \rightrightarrows V$ in~\eqref{G} has closed graph. Moreover, for any point $(\bar w, \bar u, \bar v)$ belonging to the graph of $G$, the limiting coderivative $D^*G(\bar w, \bar u, \bar v) : V \rightrightarrows W \times U$ is given by the formula
\begin{align}\label{coderi_G}
D^*G(\bar w, \bar u, \bar v)(v')  = \begin{cases}
	\{(0, 0)\}, \quad &\mbox{if } - v' \in N\left(\bar v; C \times Q \times \{0_{\mathbb{R}^l}\} \right)\\
	\emptyset, \quad &\mbox{otherwise,}
\end{cases}
\end{align} where $v'\in V$ is arbitrarily chosen.
\end{lemma}
\begin{proof}
Since the sets $C$ and $Q$ are closed, $\gph G = W \times U \times (C \times Q \times \{0_{\mathbb{R}^l}\})$ is a closed set in $W \times U \times V$. For any $(\bar w, \bar u, \bar v) \in \gph G$, the value of the limiting coderivative $D^*G(\bar w, \bar u, \bar v)$ at each $v'\in V$ can be computed as follows:
\begin{align*}
 & D^*G(\bar w, \bar u, \bar v)(v')\\
&  = \{(w', u') \in W \times U  :  (w', u', -v')\in N \big((\bar w, \bar u, \bar v); \gph G\big)\}\\
& = \{(w', u') \in W \times U  :  (w', u', -v')\in N \big((\bar w, \bar u, \bar v); W\! \times\! U \times (C \! \times\! Q \times\! \{0_{\mathbb{R}^l}\})\big)\}\\
& = \{(w', u') \in W \times U  :  (w', u', -v')\in \{0_W\} \times \{0_U\} \times N \big( \bar v; C \times Q \times \{0_{\mathbb{R}^l}\}\big)\}.
\end{align*}
This proves~\eqref{coderi_G}.
\end{proof}

Now, we are in a position to prove Theorem~\ref{Lipchitz-like_S_thm}.

\medskip
\textbf{Proof of Theorem~\ref{Lipchitz-like_S_thm}.}\ \, Given any $(\bar A, \bar B) \in \mathbb{R}^{l\times n} \times \mathbb{R}^{l\times m}$ and $(\bar x, \bar y) \in S (\bar A,\bar B)$, we put $\bar w=(\bar A, \bar B)$, $\bar u = (\bar x, \bar y)$, and $\bar v= -f(\bar w, \bar u)=(\bar x, \bar y, - \bar A \bar x + \bar B \bar y).$ On basis of the representation~\eqref{S_solution map}, we will apply Theorem~\ref{Thm. 4.2-Mor-2004-JOGO} to study the Lipschitz-likeness of $S$ at the point $(\bar w, \bar u)\in\gph S$. Note that $f$ is strictly differentiable at $(\bar w, \bar u)$ by Lemma~\ref{lemma_f} and  $G$ has closed graph by Lemma~\ref{lemma_G}.

To go furthermore, we need to explore the implication~\eqref{Lipschitz-like condition}, which, in our setting, reads as follows:
\begin{equation*}
	\big[(w', 0) \in \nabla f(\bar w, \bar u)^*(v')+D^*G(\bar w, \bar u, \bar v)(v') \big]\Longrightarrow [w'=0,\, v'=0].
\end{equation*}
By the formula~\eqref{coderi_G} in Lemma~\ref{lemma_G}, this implication means that: \textit{If $w'=(A', B') \in\!W$ and $v'=(x', y', z') \in V$ are such that
\begin{equation}\label{identity_1anew}
	- v' \in N\left(\bar v; C \times Q \times \{0_{\mathbb{R}^l}\} \right)
\end{equation}
and
\begin{align}\label{identity_2}
	(w', 0) = \nabla f(\bar w, \bar u)^*(v'),
\end{align}
then one must have $w' = 0$ and $v'=0$}.

As a matter of fact, since
$N\left(\bar v; C \times Q \times \{0_{\mathbb{R}^l}\} \right) = N(\bar x, C) \times N(\bar y, Q) \times \mathbb{R}^l,$ the inclusion~\eqref{identity_1anew} holds if and only if
\begin{equation}\label{identity_1a}
	x' \in -N(\bar x, C)\quad {\rm and}\quad  y' \in -N(\bar y, Q).
\end{equation} Clearly,~\eqref{identity_2} is equivalent to the condition
\begin{align*}
	\langle w', w \rangle =  \langle \nabla f(\bar w, \bar u)^*(v'), (w, u) \rangle, \quad \forall w = (A, B) \in W, \, u =(x, y) \in U,
\end{align*}
which means that
\begin{align*}
	\langle w', w \rangle = \langle v', \nabla f(\bar w, \bar u) (w, u) \rangle,\quad \forall w = (A, B) \in W, \, u =(x, y) \in U.
\end{align*}
Thus, using~\eqref{deri_f}, we can rewrite~\eqref{identity_2} equivalently as
\begin{align}\label{identity_2a}
	\langle w', w \rangle = -\langle x', x\rangle - \langle y', y \rangle+ \langle z', \bar A x - \bar B y + A \bar x - B \bar y \rangle
\end{align}
for all $w = (A, B) \in W$ and $u =(x, y) \in U.$

The two assertions of the theorem can be proved as follows.

\smallskip
(\textit{Sufficiency}) Suppose that the condition~\eqref{Lipchitz-like_S_con} is fulfilled. If we can show that the conditions~\eqref{identity_1a} and~\eqref{identity_2a} yield $w' = 0$ and $v'=0$, where $v'=(x', y', z')$, then from the first assertion of Theorem~\ref{Thm. 4.2-Mor-2004-JOGO} and the  above preparations we can deduce that the map $S(\cdot)$ is Lipschitz-like at $(\bar w, \bar u)$.

Substituting $w=(A, B)=(0, 0)$ and $u=(x, y)=(x, 0)$ into the equality \eqref{identity_2a}, we get
\begin{equation}\label{identity_2b}
	\langle x', x\rangle = \langle z', \bar A x\rangle,\quad \forall x\in \mathbb{R}^n.
\end{equation}
It follows that $x' = \bar A^{\rm T}z'$. So, by the first inclusion in~\eqref{identity_1a} one has $\bar A^{\rm T}z' \in -N(\bar x, C)$, i.e.,
\begin{equation}\label{identity_3}
	z' \in \big(\bar A^{\rm T}\big)^{-1}\big(-N(\bar x; C)\big).
\end{equation}
Similarly, substituting $w=(A, B)=(0, 0)$ and $u=(x, y)=(0,y)$ into the equality  \eqref{identity_2a}, we obtain
\begin{equation}\label{identity_2c}
	\langle y', y\rangle = -\langle z', \bar B y\rangle, \quad \forall y\in \mathbb{R}^m.
\end{equation}
This mounts to saying that
$y' =-\bar B^{\rm T}z'$.  Hence,  by virtue of the second inclusion in~\eqref{identity_1a} one has $\bar B^{\rm T}z' \in N(\bar y, Q)$. Therefore,
\begin{equation}\label{identity_4}
	z' \in \big(\bar B^{\rm T}\big)^{-1}\big(N(\bar y; Q)\big).
\end{equation}
Thanks to the assumption \eqref{Lipchitz-like_S_con} together with \eqref{identity_3} and~\eqref{identity_4}, we get $z'=0$. Thus,  using~\eqref{identity_2b} and~\eqref{identity_2c}, we can infer that $x'=0$ and $y' = 0$. Finally, since $x' = 0$, $y'=0$, and $z'=0$, by~\eqref{identity_2a} we can deduce that $w'=0$.

Thus, we have proved that~\eqref{Lipchitz-like_S_con} is a sufficient condition for the map $S(\cdot)$ to be Lipschitz-like at $(\bar w, \bar u)$.

(\textit{Necessity}) Suppose that $S(\cdot)$ is Lipschitz-like at $(\bar w, \bar u)=\big((\bar A, \bar B),(\bar x, \bar y)\big)$ \textit{and the additional assumption $\bar u=(\bar x, \bar y) \neq (0, 0)$ is satisfied}. To apply the second assertion of Theorem~\ref{Thm. 4.2-Mor-2004-JOGO}, we have to prove that the regularity condition~\eqref{regularity condition} holds.  In our setting, this condition is reformulated as:
\begin{equation}\label{regularityS}
	\big[(0, 0) \in \nabla f(\bar w, \bar u)^*(v')+D^*G(\bar w, \bar u, \bar v)(v') \big]\Longrightarrow [v'=0].
\end{equation}  Let $v'=(x', y', z') \in V$ be such that $(0, 0) \in \nabla f(\bar w, \bar u)^*(v')+D^*G(\bar w, \bar u, \bar v)(v').$
According to~\eqref{coderi_G}, this inclusion means that
$- v' \in N\left(\bar v; C \times Q \times \{0_{\mathbb{R}^l}\} \right)$ and
\begin{equation}\label{regularityS_2}
	(0, 0) = \nabla f(\bar w, \bar u)^*(v').
\end{equation}
Since  $\bar u\neq (0,0)$, the operator $\nabla f(\bar w, \bar u)^* : V \to W \times U$ is injective by Lemma~\ref{lemma_f}. Hence, the equality~\eqref{regularityS_2} yields $v'=(0, 0, 0)$. This proves the regularity condition~\eqref{regularityS}. Now since $S(\cdot)$ is Lipschitz-like at $(\bar w, \bar u)$, by the second assertion of Theorem~\ref{Thm. 4.2-Mor-2004-JOGO} we obtain~ \eqref{Lipschitz-like condition}. Next, the preparations given before the above proof of ``sufficiency'' tell us that~\eqref{Lipschitz-like condition} is equivalent to the following implication:
\begin{equation}\label{standing_assumption}\left[w' \in W\ {\rm and}\ v'=(x', y', z') \in V\  {\rm satisfy}\ \eqref{identity_1a}\ {\rm and}\ \eqref{identity_2a}\right]\ \Longrightarrow\ \left[w' = 0,\ \ v'=0\right].
\end{equation}

Using~\eqref{standing_assumption}, we will show that the property~\eqref{Lipchitz-like_S_con} is valid. Take any vector $z'$ with
\begin{align}\label{z'} z'\in \big(\bar A^{\rm T}\big)^{-1}\big(-N(\bar x; C)\big) \cap \big(\bar B^{\rm T}\big)^{-1}\big(N(\bar y; Q)\big).
\end{align}
Then, there exist $x' \in -N(\bar x; C)$ and $y' \in - N(\bar y; Q)$ such that $\bar A^{\rm T}z'=x'$ and $\bar B^{\rm T}z'=-y'$. It follows that $\langle x', x\rangle = \langle z', \bar A x\rangle$ for all $x\in \mathbb{R}^n$ and $\langle y', y\rangle = -\langle z', \bar B y\rangle$ for all $y\in \mathbb{R}^m$, which yields
\begin{align*}
	\langle x', x\rangle + \langle y', y\rangle =  \langle z', \bar A x\rangle -\langle z', \bar B y\rangle,\ \; \forall x\in \mathbb{R}^n,\, \forall y\in \mathbb{R}^m.
\end{align*}
In other words, one has
\begin{align}\label{equality0}
	0= -\langle x', x\rangle - \langle y', y\rangle + \langle z', \bar A x - \bar B y\rangle,\ \; \forall x\in \mathbb{R}^n,\,  \forall y\in \mathbb{R}^m.
\end{align}

{\sc Claim 1:} \textit{There exists $w'=(A',B')\in W$ such that $w'$ and $(x', y', z')$ satisfy~\eqref{identity_2a}}.

Indeed, suppose that $\bar x = (\bar x_1, \bar x_2, \dots, \bar x_n)$, $\bar y = (\bar y_1, \bar y_2, \dots, \bar y_m)$, and $z' = (z'_1, z'_2, \dots, z'_l)$. Choose $w'=(A',B')\in W$, where
\begin{equation*}\label{A'-B'}
	A':=\left (\begin{array}{c c c c}
		z'_1\bar x_1 & z'_1 \bar x_2 &\cdots &z'_1\bar x_n\\
		z'_2\bar x_1 & z'_2 \bar x_2 &\cdots &z'_2\bar x_n\\
		\vdots & \vdots &\cdots &\vdots\\
		z'_l\bar x_1 & z'_l \bar x_2 &\cdots &z'_l\bar x_n
	\end{array}\right),\quad
	B':=\left (\begin{array}{c c c c}
		-z'_1\bar y_1 & -z'_1 \bar y_2 &\cdots &-z'_1\bar y_m\\
		-z'_2\bar y_1 & -z'_2 \bar y_2 &\cdots &-z'_2\bar y_m\\
		\vdots & \vdots &\cdots &\vdots\\
		-z'_l\bar y_1 & -z'_l \bar y_2 &\cdots &-z'_l\bar y_m
	\end{array}\right).
\end{equation*}
Fix any $w = (A, B) \in W$ with
\begin{equation*}
	A=\left (\begin{array}{c c c c}
		a_{11} & a_{12} &\cdots &a_{1n}\\
		a_{21} & a_{22} &\cdots &a_{2n}\\
		\vdots & \vdots &\cdots &\vdots\\
		a_{l1}& a_{l2} &\cdots &a_{ln}
	\end{array}\right),\quad
	B=\left (\begin{array}{c c c c}
		b_{11} & b_{12} &\cdots &b_{1m}\\
		b_{21} & b_{22} &\cdots &b_{2m}\\
		\vdots & \vdots &\cdots &\vdots\\
		b_{l1}& b_{l2} &\cdots &b_{lm}
	\end{array}\right).
\end{equation*}
Then, we have $$\langle A', A\rangle = \displaystyle \sum_{i=1}^l\sum_{j=1}^{n} (z'_i a_{ij} \bar x_j) = \displaystyle \sum_{i=1}^l z'_i \left(\displaystyle \sum_{j=1}^n a_{ij} \bar x_j\right) = \langle z', A \bar x \rangle$$
and
$$\langle B', B\rangle = \displaystyle \sum_{i=1}^l \sum_{j=1}^m (-z'_i b_{ij} \bar y_j) = - \displaystyle \sum_{i=1}^l z'_i \left(\displaystyle \sum_{j=1}^{m} b_{ij} \bar y_j\right) = - \langle z', B \bar y \rangle.$$
It follows that $\langle w', w \rangle=\langle A', A\rangle+ \langle B', B\rangle = \langle z', A \bar x - B \bar y\rangle$. As $w  = (A, B) \in W$ can be chosen arbitrarily, this and~\eqref{equality0} imply that
\begin{align*}
	\langle w', w \rangle = -\langle x', x\rangle - \langle y', y\rangle + \langle z', \bar A x - \bar B y\rangle + \langle z', A \bar x - B \bar y\rangle
\end{align*}
for all $w = (A, B) \in W$ and $u=(x, y)\in U$. In other words, $w'$ and $(x', y', z')$ satisfy~\eqref{identity_2a}. Consequently, Claim 1 is true.

Since $x' \in -N(\bar x; C)$ and $y' \in - N(\bar y; Q)$ by the above constructions, condition~\eqref{identity_1a} holds. Hence, for $v':=(x', y', z')$, using Claim 1 and the implication~\eqref{standing_assumption}, we obtain $w' =0$ and $v'=0$. In particular, one has $z'=0$. Since the vector $z'$ satisfying~\eqref{z'} is arbitrary, the validity of ~\eqref{Lipchitz-like_S_con} has been proved.

The proof of the theorem is complete. $\hfill\Box$

\medskip
The additional condition $(\bar x, \bar y) \neq (0, 0)$ is essential for the validity of the second assertion of Theorem~\ref{Lipchitz-like_S_thm}. In other words, if $(\bar x, \bar y)=(0, 0)$, then it may happen that $S(\cdot)$ is Lipschitz-like at $\big((\bar A, \bar B),(\bar x, \bar y)\big)$, but~\eqref{Lipchitz-like_S_con} fails to hold. To justify the claim, let us consider the following examples.
\begin{Example}
	{\rm Let $l=m=n\geq 1$, $C=Q=\{0\}$, $\bar A\in \mathbb{R}^{n\times n}$, $\bar B \in  \mathbb{R}^{n\times n}$ be any matrices, and $(\bar x, \bar y)=(0, 0)\in\mathbb R^n\times\mathbb R^n$. Clearly, $(\bar x, \bar y)=(0, 0)\in S (\bar A, \bar B)$. In addition, $S (A, B)=\{(0,0)\}$ for all $(A, B) \in \mathbb{R}^{n\times n} \times \mathbb{R}^{n\times n}$. Since $S(\cdot)$ is a constant set-valued, it is Lipschitz-like at $\big((\bar A, \bar B),(\bar x, \bar y)\big)$. Hence~\eqref{Lipchitz-like_S_con} is invalid, because the set on its left-hand side is~$\mathbb R^n$.}
\end{Example}

\begin{Example}
	{\rm Let $l=m=n=1$, $C=(-\infty, 0]$, $Q=[0,+\infty)$. Besides, let $\bar A = (\bar \alpha)$, $\bar B=(\bar\beta)$ with $\bar\alpha>0$, $\bar\beta>0$ and take $(\bar x, \bar y)=(0, 0)\in\mathbb R\times\mathbb R$. Obviously, $(\bar x, \bar y)=(0, 0)\in S(\bar A, \bar B)$. In addition, $S (A, B)=\{(0,0)\}$ for all $(A, B)$ in a neighborhood of $(\bar A, \bar B)$. As $S(\cdot)$ is a locally constant set-valued map around $(\bar A, \bar B)$, it is Lipschitz-like at $\big((\bar A, \bar B),(\bar x, \bar y)\big)$. Since the set on the left-hand side of~\eqref{Lipchitz-like_S_con} is $(-\infty,0]$, the equality does not hold}.
\end{Example}

\section{Split Feasibility Problems}\label{SFP}\setcounter{equation}{0}
In this section, we study the stability of the solutions of the split feasibility problem, which is recalled as follows.
Let $C \subset \mathbb R^n$ and $Q \subset \mathbb R^ m$ be given nonempty closed convex sets, and let $A \in \mathbb R^{m\times n}$ be a given matrix. The problem of finding a point $x \in C$ such that $Ax \in Q$ is called the \textit{split feasibility problem} (SFP). We pay attention to stability properties of (SFP) when $A$ is subject to change, while $C, Q$ are fixed. We define the solution map $S_1(\cdot): \mathbb R^{m \times n} \rightrightarrows \mathbb R^n$ of (SFP) by
\begin{equation}\label{S1}
S_1(A)=\{x \in C \, : \, Ax \in Q\}, \quad A \in \mathbb R^{m\times n}.
\end{equation}

In the main result of this section, Theorem~\ref{Lipchitz-like_S1_thm}, we provide
sufficient and necessary conditions for the Lipschitz-likeness property of the solution map $S_1(\cdot)$ at a given point in its graph.

\begin{theorem}\label{Lipchitz-like_S1_thm}
	Let $\bar A \in \mathbb{R}^{m\times n}$ and $\bar x \in S_1(\bar A)$. If
	\begin{equation}\label{Lipchitz-like_S1_con}
		\big(\bar A^{\rm T}\big)^{-1}(-N(\bar x; C)) \cap N(\bar A \bar x; Q)=\{0\},
	\end{equation}
	then the solution map $S_1(\cdot)$ is Lipschitz-like at $(\bar A, \bar x)$. Conversely, if $S_1(\cdot)$ is Lipschitz-like at $(\bar A, \bar x)$ and if $\bar x \neq 0$, then~\eqref{Lipchitz-like_S1_con} is fulfilled.
\end{theorem}

\begin{remark}{\rm Obviously, if either $(\bar A^{\rm T})^{-1}(-N(\bar x; C)) =\{0\}$ or $N(\bar A \bar x; Q)=\{0\}$, then condition~\eqref{Lipchitz-like_S1_con} is satisfied. Thus, if either $\bar x \in \mbox{int}\, C$ and $\ker \bar A^{\rm T} = \{0\}$ or $\bar A \bar x \in \mbox{int}\, Q$, then~\eqref{Lipchitz-like_S1_con} is valid.}
\end{remark}

Let us provide two illustrating examples for Theorem~\ref{Lipchitz-like_S1_thm}.

\begin{Example}{\rm Choose $m=n=1$, $C= [-1, 1]$, and $Q=[0, +\infty)$. Let $\bar A=(\bar\alpha)$ with $\bar\alpha> 0$ and fix any $\bar x \in [0, 1]$. It is not hard to see that $\bar x \in S_1(\bar A)$. If $\bar x  \in [0, 1)$, then $\bar x \in \mbox{int}\, C$. As $\ker \bar A^{\rm T} = \{0\}$, this implies that~\eqref{Lipchitz-like_S1_con} is satisfied. Therefore, by Theorem~\ref{Lipchitz-like_S1_thm}, $S_1(\cdot)$ is Lipschitz-like at $(\bar A, \bar x)$. If $\bar x = 1$, then $\bar A \bar x = \bar\alpha\in \mbox{int}\, Q$. Thus, \eqref{Lipchitz-like_S1_con} is fulfilled. Hence, the Lipschitz-likeness of $S_1(\cdot)$ at $(\bar A, \bar x)$ is assured by Theorem~\ref{Lipchitz-like_S1_thm}.}
\end{Example}

\begin{Example}{\rm As in the previous example, we consider (SFP) with  $m=n=1$, $C= [-1, 1]$, and $Q=[0, +\infty)$. Let $\bar A=(0)$ and take any $\bar x \in [-1, 1]$. Then, one has $\bar x \in S_1(\bar A)$. Observe that $\big(\bar A^{\rm T}\big)^{-1}\big(-N(\bar x; C)\big) = \mathbb R$, while $N(\bar A\bar x; Q)  = (-\infty, 0]$. Hence, one has $$\big(\bar A^{\rm T}\big)^{-1}\big(-N(\bar x; C)\big) \cap N(\bar A\bar x; Q)   = (-\infty, 0].$$ This means that condition~\eqref{Lipchitz-like_S1_con} is invalid. Therefore, if $\bar x \neq 0$, then by Theorem~\ref{Lipchitz-like_S1_thm} we can conclude that $S_1(\cdot)$ is \textit{not} Lipschitz-like at $(\bar A, \bar x)$.}
\end{Example}

The next examples show that the assumption~$\bar x \neq 0$ is vital for the second assertion of Theorem~\ref{Lipchitz-like_S1_thm}.

\begin{Example} {\rm
 Let $m, n$ be any positive integers, $C=\{0_{\mathbb R^n}\}$, and $Q=\{0_{\mathbb R^m}\}$. Take any matrix $\bar A$ from  $\mathbb{R}^{m\times n}$ and let $\bar x = \{0_{\mathbb R^n}\}$. Clearly, $\bar x\in S_1(\bar A)$. In addition, $S_1(A)=\{0_{\mathbb R^n}\}$ for all $A \in \mathbb{R}^{m\times n}$. Since $S_1(\cdot)$ is a constant set-valued map, it is Lipschitz-like at $(\bar A, \bar x)$. But the equality~\eqref{Lipchitz-like_S1_con} is invalid, because the set on its left-hand side is~$\mathbb R^m$.	
 This shows that the nonzero solution (i.e., $\bar x \neq 0$) assumption in the necessity condition of Theorem \ref{Lipchitz-like_S1_thm} can't be removed.
}
\end{Example}

\begin{Example} {\rm
Let $m, n$ be arbitrarily positive integers, $C = \mathbb R_-^n$ and $Q = \mathbb R_+^m$, where
$$\mathbb R_-^n:=\{x = (x_1, x_2, \dots, x_n) \in \mathbb R^n \, : \, x_i \leq 0,\; \forall i \in \{1,\dots, n\} \}$$
and
$$\mathbb R_+^m:=\{y = (y_1, y_2, \dots, y_m) \in \mathbb R^m \, : \, y_i \geq 0,\; \forall i \in \{1,\dots, m\}\}$$
are the nonpositive orthant and nonnegative orthant in $\mathbb R^n$ and $\mathbb R^m$, respectively. Choose $\bar A = (\bar a_{ij}) \in \mathbb{R}^{m\times n}$ with $\bar a_{ij} > 0$ for all $i, j$ and take $\bar x = \{0_{\mathbb R^n}\}$. It is easy to see that $\bar x \in S_1(\bar A)$. Besides, $S_1(A) = \{0_{\mathbb R^n}\}$ for any matrix $A \in \mathbb{R}^{m\times n}$ close enough to $\bar A$. Thus, the solution map $S_1(\cdot)$ is Lipschitz-like at $(\bar A, \bar x)$.
Since $N(\bar A \bar x; Q)=\mathbb R^m_-$ and $\big(\bar A^{\rm T}\big)(\mathbb R^m_-)\subset -N(\bar x; C)$, the set on the left-hand side of~\eqref{Lipchitz-like_S1_con} is~$\mathbb R^m_-$. So, the equality is not fulfilled.}
\end{Example}

\textbf{Proof of Theorem~\ref{Lipchitz-like_S1_thm}.}\ \, Fix any $\bar A \in \mathbb{R}^{m\times n}$ and $\bar x \in S_1(\bar A)$. We divide the proof into two parts: the ``sufficiency" is derived from the first assertion of Theorem~\ref{Lipchitz-like_S_thm}, while the ``necessity'' requires a series of independent arguments. (Note that the second assertion of Theorem~\ref{Lipchitz-like_S_thm} is based on the assumption that \textit{both matrices $A$ and $B$ are subject to perturbations}. Meanwhile, for (SFP), only the matrix $A$ is perturbed.)

\smallskip
(\textit{Sufficiency}) Suppose that condition~\eqref{Lipchitz-like_S1_con} holds. Put $\bar y = \bar A \bar x$ and let $\bar B$ be the identity matrix $I_m$ in $\mathbb R^{m \times m}$. As $\bar x \in S_1(\bar A)$, one has $(\bar x, \bar y) \in S(\bar A, \bar B)$, where $S(\cdot)$ is the solution map of (SEP) with $l:=m$. Besides, since~\eqref{Lipchitz-like_S1_con} holds, condition~\eqref{Lipchitz-like_S_con} is fulfilled. Thus, by Theorem~\ref{Lipchitz-like_S_thm}, the map $S(\cdot)$ is Lipschitz-like at $\big((\bar A, \bar B),(\bar x, \bar y)\big)$. Therefore, we can find constants $\delta_1 >0$, $\delta_2>0$ , and $\ell >0$ such that
\begin{align*}
	S(A', B') \cap  {B}\big((\bar x, \bar y), \delta_2\big) \subset S(A, B) + \ell \|(A', B') -(A, B)\| \Bar {\B}_{\mathbb R^n\times \mathbb R^m}
\end{align*}
for any $(A', B'), (A, B) \in  \B\big((\bar x, \bar y), \delta_1\big)$. In particular, one has
\begin{align}\label{lip1}
	S(A', I_m) \cap \B\big((\bar x, \bar y), \delta_2\big) \subset S(A, I_m) + \ell \| A' - A\| \Bar {\B}_{\mathbb R^n\times \mathbb R^m},
\end{align}
for all $A', A \in \B(\bar A, \delta_1).$
Choose $\delta_1' \in (0, \delta_1)$ and $\delta_2'\in (0,\delta_2)$ as small as \begin{equation}\label{small_delta} (1+\|\bar A\|) \delta_2' + \delta_1' (\|\bar x\| +\delta_2) \leq \delta_2.\end{equation}
We claim that
\begin{equation}\label{lip2}
	S_1(A') \cap \B (\bar x, \delta_2') \subset S_1(A) +\ell \|A'-A\| \Bar {\B}_{\mathbb R^n}, \quad \forall A', \, A \in \B(\bar A, \delta_1').
\end{equation}
To prove this, let $A', \, A \in B(\bar A, \delta_1')$ be given arbitrarily. Take an $x' \in S_1(A') \cap \B (\bar x, \delta_2')$ and put $y'=A'x'$. Then we have
\begin{equation}\label{lip3}
(x', y') \in	S(A', I_m).
\end{equation}
 Besides, it holds that
\begin{align*}
\| (x', y') - (\bar x, \bar y)\| & =  \|x'-\bar x\| + \|A'x' - \bar A \bar x\| \\
& \leq \|x'-\bar x\| + \|A'x' - \bar A x'\| + \| \bar A x'- \bar A \bar x\| \\
& \leq (1+\|\bar A\|)\| x' - \bar x\| + \|A' -\bar A\| \|x'\|.
\end{align*}
Combining this with the facts that $A',\, A \in \B(\bar A, \delta_1')$ and $x' \in  \B (\bar x, \delta_2')$ yields $$\| (x', y') - (\bar x, \bar y)\|  \leq  (1+\|\bar A\|) \delta_2' + \delta_1' (\|\bar x\| +\delta_2') <  (1+\|\bar A\|) \delta_2' + \delta_1' (\|\bar x\| +\delta_2).$$ Thus, by~\eqref{small_delta} we get  $(x', y') \in \B\big((\bar x, \bar y), \delta_2\big).$ This and~\eqref{lip3} imply that $$(x', y') \in	S(A', I_m) \cap \B\big((\bar x, \bar y), \delta_2\big).$$ So, by~\eqref{lip1} we can find $(x, y) \in S(A, I_m)$ and $(u,v) \in \Bar {\B}_{\mathbb R^n\times \mathbb R^m}$ such that $$(x', y') = (x, y) + \ell \|A' - A\| (u,v).$$
Particularly, we get
\begin{equation}\label{represent1}
x' = x + \ell \|A' - A\|u.
\end{equation}
Since $(x, y) \in S(A, I_m)$ and $(u,v) \in \Bar {\B}_{\mathbb R^n\times \mathbb R^m}$, we obtain $x \in S_1(A)$ and $u\in \Bar {\B}_{\mathbb R^n}$. Thus, the representation~\eqref{represent1} yields $x' \in S_1(A) +\ell \|A'-A\| \Bar {\B}_{\mathbb R^n}$.
The claim \eqref{lip2} is verified.

Clearly, the validity of~\eqref{lip2} for the chosen constants $\delta_1'$ and $\delta_2'$ implies that $S_1(\cdot)$ is Lipschitz-like at $(\bar A, \bar x)$.

(\textit{Necessity}) Suppose that $S_1(\cdot)$ is Lipschitz-like at $(\bar A, \bar x)$ and $\bar x \neq 0$. To show that condition~\eqref{Lipchitz-like_S1_con} is valid, we consider the function $f_1: \mathbb R^{m\times n} \times \mathbb R^n \to \mathbb R^n \times \mathbb R^m$ defined by
\begin{equation}\label{f1}
	f_1(A, x) :=(-x, -A x), \quad (A, x) \in \mathbb R^{m\times n} \times \mathbb R^n
\end{equation}
and the constant set-valued map $G_1: \mathbb R^{m\times n} \times \mathbb R^n \rightrightarrows \mathbb R^n \times \mathbb R^m$ with
\begin{equation}\label{G1}
	G_1(A, x):=C \times Q, \quad (A, x) \in \mathbb R^{m\times n} \times \mathbb R^n.
\end{equation} Clearly, formula~\eqref{S1} can be rewritten as
\begin{equation*}
	S_1(A) = \{x \in \mathbb R^n \, : \,  0 \in f_1(A, x) + G_1(A, x)\}, \quad A \in \mathbb R^{m\times n}.
\end{equation*}
Consequently, $S_1(\cdot)$ is the solution map of the parametric generalized equation $$0 \in f_1(A, x) + G_1(A, x),$$ where $A$ is a parameter. To employ the second assertion of Theorem~\ref{Thm. 4.2-Mor-2004-JOGO} for $S_1(\cdot)$, we have to prove some auxiliary results.

{\sc Claim 1.} \textit{The function $f_1$ is strictly differentiable at any $(\bar A, \bar x) \in \mathbb R^{m\times n} \times \mathbb R^n$. The derivative $\nabla f_1(\bar A, \bar x): \mathbb R^{m\times n} \times \mathbb R^n \to \mathbb R^n \times \mathbb R^m$ of $f_1$ at $(\bar A, \bar x)$ is given by
	\begin{equation}\label{deri_f1}
		\nabla f_1(\bar A, \bar x)(A, x) = (-x, - \bar A x  - A \bar x), \quad (A, x) \in \mathbb R^{m\times n} \times \mathbb R^n.
	\end{equation}
	Moreover, if $\bar x \neq 0$, then  the operator $\nabla f_1(\bar A, \bar x)$ is surjective and thus its adjoint operator  $\nabla f_1(\bar A, \bar x)^* : \mathbb R^n \times \mathbb R^m \to \mathbb R^{m\times n} \times \mathbb R^n$ is injective.}

	Indeed, fixing any $(\bar A, \bar x) \in \mathbb R^{m\times n} \times \mathbb R^n$, we consider the continuous linear operator $T_1(\bar A, \bar x): \mathbb R^{m\times n} \times \mathbb R^n \to \mathbb R^n \times \mathbb R^m$ given by
	\begin{equation}\label{T1}
		T_1(\bar A, \bar x) (A, x):= (-x, - \bar A x  - A \bar x), \quad (A, x) \in \mathbb R^{m\times n} \times \mathbb R^n.
	\end{equation}
 Set
 $$L_1:=\displaystyle\lim_{(A, x)\rightarrow (\bar A,\bar x)}\dfrac{f_1(A, x)-f_1(\bar A,\bar x)
 - T_1(\bar A,\bar x)((A, x)-(\bar A,\bar x))}{\|(A, x)-(\bar A,  \bar x)\|}.$$
 By~\eqref{f1} one has
	\begin{align*}
	L_1	&=\displaystyle\lim_{(A, x)\rightarrow (\bar A,\bar x)}\Big[
			\dfrac{(-x, -Ax)- (-\bar x, -\bar A \bar x) - \big( -(x -\bar x), - \bar A (x-\bar x) - (A -\bar A)\bar x \big)}
        {\|A-\bar A\| +\| x-\bar x\|} \Big] \\
		& =\displaystyle\lim_{(A, x)\rightarrow (\bar A,\bar x)}\bigg (0,\dfrac{-(A -\bar A)(x-\bar x)}{\|A-\bar A\|+\|x-\bar x\|} \bigg).
	\end{align*}
	Because
	\begin{align*}
		\dfrac{\|-(A - \bar A)(x-\bar x)\|}{\|A-\bar A\|+\|x-\bar x\|}
\leq \dfrac{\|A - \bar A\| \|x-\bar x\|}{\|A-\bar A\|+\| x-\bar x\|}\leq \|x-\bar x\|\to 0
	\end{align*}
 as $x\to \bar x$,  one immediately gets $L_1 = 0$. This proves that $f_1$ is Fr\'echet differentiable  at $(\bar A, \bar x)$ and $\nabla f_1(\bar A, \bar x)$ coincides with the operator $T_1(\bar A, \bar x)$ in~\eqref{T1}. Since the function $\Phi_1:\mathbb R^{m\times n} \times \mathbb R^n\to L(\mathbb R^{m\times n} \times \mathbb R^n,\mathbb R^n \times \mathbb R^m)$ with $\Phi_1(\widetilde A,\widetilde x):=T_1(\widetilde A, \widetilde x)$ is continuous on $\mathbb R^{m\times n} \times \mathbb R^n$, $f_1$ is strictly differentiable at $(\bar A, \bar x)$ (see~\cite[p.~19]{B-M06}) and its strict derivative is given by~\eqref{deri_f1}.
	
	To justify the last assertion of the claim, suppose that $\bar x = (\bar x_1, \bar x_2, \dots, \bar x_n)$ with $\bar x_{\tilde j} \neq 0$ for some index $\tilde j \in \{1, 2, \dots, n\}$. The surjectivity of  $\nabla f_1(\bar A, \bar x)$ will be proved by showing that,  for any given point $(u, v) \in \mathbb R^n \times \mathbb R^m$, there exists $(A, x) \in \mathbb R^{m\times n} \times \mathbb R^n$ satisfying
	\begin{equation}\label{equation_new2}
		\nabla f_1(\bar A, \bar x)(A, x) = (u, v).
	\end{equation}
	By~\eqref{deri_f1}, we have $\nabla f_1(\bar A, \bar x)(A, x) = (u, v)$ if and only if $(-x, - \bar A x  - A \bar x) = (u, v)$. Thus, relation~\eqref{equation_new2} requires $x = -u$ and $A$ is such a matrix that
	\begin{equation}\label{identity_5new}
		A \bar x = \bar A u - v.
	\end{equation}
	Put $\widetilde v=\bar A u - v$ and write $\widetilde v=(\widetilde v_1, \widetilde v_2, \dots, \widetilde v_m)$.
Clearly,~\eqref{identity_5new} is equivalent to
	\begin{align}\label{identity_5}
		A \bar x = \widetilde v.
	\end{align}
	Define $A = (a_{ij}) \in \mathbb{R}^{m\times n}$ by setting $a_{i\tilde j}
= \dfrac{\widetilde v_i}{\bar x_{\tilde j}}$ for $i \in \{1, \dots, m\}$, $a_{ij} = 0$ for $i \in \{1, \dots, m\}$ and $j \in \{1, \dots, n\} \setminus \{\tilde j\}$. Then we have
	\begin{align*}
		A \bar x = \left (\begin{array}{c} a_{11}\bar x_1
			+ a_{12} \bar x_2 + \cdots + a_{1n} \bar x_n \\
			a_{21}\bar x_1 + a_{22} \bar x_2 + \cdots + a_{2n} \bar x_n\\
			\vdots\\
			a_{m1}\bar x_1 + a_{m2} \bar x_2 + \cdots + a_{mn} \bar x_n
		\end{array}\right) = \left(\begin{array}{c}\widetilde v_1\\
			\widetilde v_2\\
			\vdots\\
			\widetilde v_m \end{array}\right).
	\end{align*}
	So, this matrix $A$ satisfies~\eqref{identity_5} and the surjectivity of $\nabla f_1(\bar A, \bar x)$ is proved.
The injectivity of $\nabla f_1(\bar A, \bar x)^*$ is then guaranteed by~\cite[Lemma 1.18]{B-M06}.

Since the proof of the next claim is similar to that of  Lemma~\ref{lemma_G}, we omit it.

{\sc Claim 2.} \textit{The set-valued map $G_1: \mathbb R^{m\times n} \times \mathbb R^n \rightrightarrows \mathbb R^n \times \mathbb R^m$ defined in~\eqref{G1} has closed graph and, for any given point $\big((\bar A, \bar x), (\bar u, \bar v)\big)$ belonging to the graph of $G_1$, the limiting coderivative $D^*G_1\big((\bar A, \bar x), (\bar u, \bar v)\big):\mathbb R^n \times \mathbb R^m\rightrightarrows \mathbb R^{m\times n} \times \mathbb R^n$ is given by
	\begin{equation}\label{coderi_G1}
		D^*G_1((\bar A, \bar x), (\bar u, \bar v))(u', v')=
		\begin{cases}
			\{(0, 0)\}, \quad &\mbox{if } u' \in -N(\bar u; C), \ v' \in -N(\bar v; Q)\\
			\emptyset, & \mbox{otherwise},
		\end{cases}
	\end{equation}
	where $(u', v')$ is an arbitrary point in $\mathbb R^n \times \mathbb R^m$.}

Since $f_1$ is strictly differentiable at $(\bar A, \bar x)$ by Claim~1 and  $G_1$ has closed graph by Lemma~\ref{lemma_G},  to apply Theorem~\ref{Thm. 4.2-Mor-2004-JOGO}, we need to prove that condition~\eqref{regularity condition} holds.  In our setting, the latter reads as follows:
\begin{equation}\label{regularityS1}
	\big[(0, 0) \in \nabla f_1(\bar A, \bar x)^*(u', v') + D^*G_1\big((\bar A, \bar x), (\bar u, \bar v)\big) (u', v') \big]\Longrightarrow [(u', v')=(0, 0)],
\end{equation}
where $(\bar u, \bar v) = -f_1(\bar A, \bar x) = (\bar x, \bar A \bar x).$  Suppose that $(u', v') \in \mathbb R^n \times \mathbb R^m$ satisfies
\begin{equation*}
	(0, 0) \in \nabla f_1(\bar A, \bar x)^*(u', v') + D^*G_1\big((\bar A, \bar x), (\bar u, \bar v)\big) (u', v').
\end{equation*}
Then, by invoking formula~\eqref{coderi_G1}, we get
\begin{equation}\label{cases}
		\nabla f_1(\bar A, \bar x)^*(u', v') = (0, 0).
\end{equation}
As $\bar x \neq 0$, the operator $\nabla f_1(\bar A, \bar x)^* : \mathbb R^n \times \mathbb R^m \to \mathbb R^{m\times n} \times \mathbb R^n$ is injective by Claim~1. So, it follows from~\eqref{cases} that $(u', v') = (0, 0)$. We have thus verified condition~\eqref{regularityS1}.
We next turn to prove (\ref{Lipchitz-like_S1_con}).

 Since $S_1(\cdot)$ is Lipschitz-like at $(\bar A, \bar x)$, by the second assertion of Theorem~\ref{Thm. 4.2-Mor-2004-JOGO}, we obtain~\eqref{Lipschitz-like condition}. In our setting, condition~\eqref{Lipschitz-like condition} can be restated as follows:
\begin{equation*}
	\big[(A', 0) \in \nabla f_1(\bar A, \bar x)^*(u', v') + D^*G_1\big((\bar A, \bar x), (\bar u, \bar v)\big) (u', v') \big]\Longrightarrow [A'=0,\, u' = 0, \, v'=0].
\end{equation*}
Thanks to~\eqref{coderi_G1}, this implication means that if $A' \in \mathbb R^{m \times n}$ and $(u', v') \in \mathbb R^n \times \mathbb R^m$ are such that
\begin{equation}\label{inclusion_1}
	u' \in - N(\bar x; C),\ \;  v' \in - N(\bar A \bar x; Q)
\end{equation}
and
\begin{equation}\label{inclusion_2}
	(A', 0) = \nabla f_1(\bar A, \bar x)^*(u', v'),
\end{equation}
then one must have $A' = 0$, $u' = 0$, and $v' = 0$. Equality~\eqref{inclusion_2} is equivalent to
\begin{align*}
	\langle A', A \rangle &= \langle \nabla f_1(\bar A, \bar x)^*(u', v'), (A, x) \rangle \\
	& = \langle (u', v'), \nabla f_1(\bar A, \bar x) (A, x) \rangle, \quad \forall (A, x) \in \mathbb R^{m \times n} \times \mathbb R^n.
\end{align*}
Therefore, by~\eqref{deri_f1} we can rewrite~\eqref{inclusion_2} as
\begin{equation}\label{inclusion_2a}
	\langle A', A \rangle = - \langle  u', x \rangle -  \langle v',  \bar A x + A \bar x \rangle, \quad \forall (A, x) \in \mathbb R^{m \times n} \times \mathbb R^n.
\end{equation}
Thus,~\eqref{Lipschitz-like condition} means the following:
\begin{equation}\label{standing_assumption1}
	\left[A'\ {\rm and}\  (u', v')\  {\rm satisfy}\  \eqref{inclusion_1}\ {\rm and}\  \eqref{inclusion_2a}\right] \Longrightarrow \left[A' = 0,\ u' =0,\  v'=0\right].
\end{equation}
Now, to complete the proof, we will use~\eqref{standing_assumption1} to show that~\eqref{Lipchitz-like_S1_con} holds.

Take any $v'\in \big(\bar A^{\rm T}\big)^{-1}\big(-N(\bar x; C)\big) \cap N\big(\bar A \bar x; Q\big)$. Then, $v' \in N\big(\bar A \bar x; Q\big)$ and there exists $u' \in -N(\bar x; C)$ such that $\bar A^{\rm T}(v') =u'$. On the one hand, as $\bar A^{\rm T}(v')-u'=0$, we get
\begin{equation}\label{inclusion_3}
	0 = - \langle  u', x \rangle - \langle -v',  \bar A x \rangle, \quad \forall x \in \mathbb R^n.
\end{equation}
On the other hand, suppose that $\bar x = (\bar x_1, \bar x_2, \dots, \bar x_n)$ and $v' = (v'_1, v'_2, \dots, v'_m)$. Consider the matrix
\begin{equation}\label{A'}
	A':=\left (\begin{array}{c c c c}
		v'_1\bar x_1 & v'_1 \bar x_2 &\cdots &v'_1\bar x_n\\
		v'_2\bar x_1 & v'_2 \bar x_2 &\cdots &v'_2\bar x_n\\
		\vdots & \vdots &\cdots &\vdots\\
		v'_m\bar x_1 & v'_m \bar x_2 &\cdots &v'_m\bar x_n
	\end{array}\right) \in \mathbb R^{m \times n}.
\end{equation}
Given any $A \in\mathbb R^{m \times n}$ with
\begin{equation*}
	A=\left (\begin{array}{c c c c}
		a_{11} & a_{12} &\cdots &a_{1n}\\
		a_{21} & a_{22} &\cdots &a_{2n}\\
		\vdots & \vdots &\cdots &\vdots\\
		a_{m1}& a_{m2} &\cdots &a_{mn}
	\end{array}\right),
\end{equation*} we have $$\langle A', A\rangle = \displaystyle \sum_{i=1}^m\sum_{j=1}^{n} v'_i a_{ij} \bar x_j = \displaystyle \sum_{i=1}^m v'_i \left(\displaystyle \sum_{j=1}^{n} a_{ij} \bar x_j\right) = \langle v', A \bar x \rangle.$$
Combining this with~\eqref{inclusion_3}, we get
\begin{equation*}
	\langle A', A\rangle =  - \langle  u', x \rangle -  \langle - v',  \bar A x + A \bar x \rangle
\end{equation*}
for any $(A, x)\in \mathbb R^{m \times n} \times \mathbb R^n$. This means that $A'$ given by~\eqref{A'} and $(u',\widetilde{v}')$ with $ \widetilde{v}':=-v'$ satisfy \eqref{inclusion_2a}, where $\widetilde{v}'$ plays the role of $v'$. Thus, keeping in mind that \eqref{inclusion_1}, where $\widetilde{v}'$ plays the role of $v'$, is satisfied because $u' \in -N(\bar x; C)$ and $\widetilde{v}'\in -N\big(\bar A \bar x; Q\big)$, we then derive from~\eqref{standing_assumption1} that $A' = 0$, $u'= 0$, and $\widetilde{v}'=0$. Hence, $v'=0$. Because $v'\in \big(\bar A^{\rm T}\big)^{-1}\big(-N(\bar x; C)\big) \cap N\big(\bar A \bar x; Q\big)$ is arbitrary, we have proved~\eqref{Lipchitz-like_S1_con}.

The proof of the theorem is complete. $\hfill\Box$

\section{Conclusion}\label{Sect_Conclusions}

Previous work on split equality problems (SEPs) and split feasibility problems (SFPs) all focused on algorithmic methods for
finding a solution. In this paper, however, we have investigated the stability of solutions of SEPs and SFPs
with respect to perturbations of the matrices $A$ and $B$ that define SEPs and SFPs.
[To the best of our knowledge, this is the first-time try in the literature of the study of SEPs and SFPs so far.]
By making use of concepts and tools from set-valued and variational analysis
(cf. Theorem \ref{Thm. 4.2-Mor-2004-JOGO} or \cite[Theorem 4.2(ii)]{Mor04}),
we have successfully established necessary and sufficient conditions for the solution map $S$ of SEP
(defined in \eqref{S}) and the solution map $S_1$ of SFP (defined in \eqref{S1})
to be Lipschitz-like (Theorems \ref{Lipchitz-like_S_thm} and \ref{Lipchitz-like_S1_thm}).
It is an interesting problem of the stability of the solution maps $S$ and $S_1$ of  SEP and SFP, respectively,
with respect to the constraint sets $C$ and $Q$. This problem remains open and is worth of further study.

\section*{Acknowledgements}
This work has been co-funded by the European Union (European Regional Development Fund EFRE, fund number: STIIV-001), the National Natural Science Foundation of China (grant number U1811461), the Australian Research Council/Discovery Project (grant number DP200100124), and the Vietnam Academy of Science and Technology (project code NCXS02.01/24-25). Vu Thi Huong and Nguyen Dong Yen would like to thank the Vietnam Institute for Advanced Study in Mathematics for hospitality during their recent stay.

\medskip

\noindent\textbf{Author Contributions}: All authors contributed equally, read and approved the final manuscript.

\medskip
\noindent\textbf{Competing Interests}: The authors have no relevant financial or non-financial interests to disclose.

\end{document}